\documentclass[12pt]{article}

\usepackage{amsmath,amsthm,amssymb,graphicx}

\def\N{\mathbb N}
\def\N0{\mathbb N_0}

\def\real{\mathbb R}

\def\alg{\mathbb A}

\def\graph{{\mathcal G}}

\def\vertexset{{\mathcal V}}

\def\edgeset{{\mathcal E}}
\def\tree{{\mathcal T}}
\def\treebar{\overline{\tree}}
\def\graphbar{\overline{\graph}}

\def\wt{\mu}

\def\cond{C}
\def\rest{R}

\def\dop{{\mathcal L}}

\def\domain {{\mathcal D}}

\title{\bf After the Explosion: \\
Dirichlet Forms and \\
Boundary Problems for \\
Infinite Graphs} 
\author{Robert Carlson \\
Department of Mathematics \\ 
University of Colorado at Colorado Springs \\
carlson@math.uccs.edu}

\newtheorem{thm}{Theorem}[section]
\newtheorem{cor}[thm]{Corollary}
\newtheorem{lem}[thm]{Lemma}
\newtheorem{prop}[thm]{Proposition}

\theoremstyle{definition}

\theoremstyle{remark}

\newcommand{\thmref}[1]{Theorem~\ref{#1}}

\newcommand{\lemref}[1]{Lemma~\ref{#1}}
\newcommand{\corref}[1]{Corollary~\ref{#1}}
\newcommand{\propref}[1]{Proposition~\ref{#1}}

 \numberwithin{equation}{section}

\begin{document}

\maketitle

\begin{abstract}

Formal Laplace operators are analyzed for a large class of resistance 
networks with vertex weights.  The graphs are completed with respect to 
the minimal resistance path metric.  Compactness and a novel connectivity
hypothesis for the completed graphs play an essential role.
A version of the Dirichlet problem is solved.
Self adjoint Laplace operators and the probability semigroups they 
generate are constructed using reflecting and absorbing conditions 
on subsets of the graph boundary.

\end{abstract}

\vskip 25pt



\newpage

\section{Introduction}

This work has its roots in the challenge of extending differential equation models
for diffusion or wave propagation from domains in Euclidean space to infinite graphs
intended to resemble biological transport systems such as the arteries of the human 
circulatory system.  Such biological systems can include enormous numbers of branching segments. 
Short time transport across the network is essential, so treelike structures with
small numbers of large edges and vast collections of microscopic edges are typical.  
Faced with such complex heterogeneous structures, 
one hopes that appropriate infinite graph models
will suggest useful structural features and robustly posed problems.

Building on an earlier 'quantum graph' analysis of such problems \cite{Carlson08}, 
this work uses infinite graph and operator theoretic methods to treat a class of 
continuous time Markov chains.  Recall that continuous time Markov chains
use a system of constant coefficient differential equations 
\begin{equation} \label{kolbeq}
\frac{dP}{dt} = QP, \quad P(0) = I.
\end{equation}
to describe the evolution of probability densities $X(t) = X(0)P(t)$ 
on a finite or countably infinite set of states. 
An associated graph may be constructed by connecting states (vertices) $i$ and $j$
with an edge if $Q_{ij} \not= 0$.

In the finite state case the solution of \eqref{kolbeq} is simply $P(t) = e^{Qt}$.
When the set of states is infinite the formal description of the operator $Q$ 
may not be adequate to determine the semigroup $e^{Qt}$, an issue known 
in probability as the problem of explosions.  
Infinite graph models inspired by biological transport 
systems will typically face the explosion problem.  
By imposing restrictions on both the form of the Markov 
chain generator $Q$ and the structure of the associated graph viewed 
as a metric space, this work provides a resolution in terms of 
'reflecting' and 'absorbing' behavior at a graph boundary. 

It will be advantageous to use the Dirichlet form theory \cite{Davies,Fukushima}. 
To that end, consider a graph $\graph $ whose edges $e = [u,v]$ are equipped with
positive weights $R(u,v)$ which are interpreted as edge length. 
With $C(u,v) = 1/R(u,v)$, a symmetric bilinear form for functions on the vertex
set is defined by
\begin{equation} \label{bilinear}
B(f,g) = \frac{1}{2}\sum_{v \in \vertexset} \sum_{u \sim v} C(u,v)(f(v) - f(u))(g(v) - g(u)).
\end{equation}
Each vertex is also given a positive weight $\wt (v)$.  
Formal semigroup generators $\Delta _{\wt}$ are defined by  
\begin{equation} \label{Lapop}
\Delta _{\wt}f(v) = \frac{1}{\wt (v)}\sum_{u \sim v} C(u,v) (f(v) - f(u)).
\end{equation}        

$R(u,v)$ is often interpreted as electrical resistance.
The electrical network analogy is treated at length in \cite{Doyle,Lyons}. 
The recent work \cite{Georg10} treats electrical currents in a context 
similar to this paper, while \cite{Georg11} treats related topological questions.
An analysis of function theory on infinite trees motivated by modeling the
human lungs is in \cite{Maury}.
With the domain of functions with finite support,
$\Delta _{\wt}$ is a symmetric operator on $l^2(\wt )$.
In contrast to this paper, other recent work \cite{Colin,Jorgensen,KL1,KL2} has stressed 
cases when this symmetric operator is essentially selfadjoint, 
and so behavior at the graph boundary is not an issue.

The vertex set $\vertexset$ of an edge weighted locally finite graph $\graph $ 
can be equipped with a metric $d(u,v)$ obtained 
by minimizing the sum of the edge lengths of paths from $u$ to $v$.
By completing this metric space we obtain a metric space $\graphbar $ 
in which one can discuss features like the graph boundary and compactness.
If $\graph $ is a tree, then distinct points of $\graphbar $ 
can be separated by deleting a suitable edge.  Generalizing this idea,
our graphs will be required to have 'weakly connected' completions, 
with the property that, for any two distinct points, 
any path joining them must include an edge from a finite set.
This generalization identifies a rich class of edge weighted graphs with 
useful topological and function theoretic properties.  

The properties of weakly connected graph completions are developed 
in the second section.  In addition to trees, arbitrary graphs with
finite volume have weakly connected completions.  This class is
also preserved if we add suitably constrained edge sequences to a graph. 
Weakly connected completions are totally disconnected metric spaces.
When also compact, these spaces are topologically stable
with respect to decrease of the metric.  The weakly connected class
will be characterized using the separation of points property for an 
algebra of 'eventually flat' functions.

The third section treats the bilinear forms, vertex weights, and
associated operators.  The choice of vertex weights typically used
for discrete time Markov chains are contrasted with weights 
making $\Delta _{\wt}$ resemble a discretized second derivative.
The bilinear form is used to construct several 'Sobolev style' Hilbert spaces
on $\graph $ whose elements extend continuously to $\graphbar$.

Two main problems are treated in the fourth section.
The first, a version of the Dirichlet problem, asks for conditions under which
continuous functions on $\partial \graphbar $ have a unique harmonic
extension to $\graph $.  An example shows that a lack of compactness 
can lead to a negative result.  Using assumptions of compactness and
weak connectivity, a general positive result is established.   
The second problem is the resolution of
the explosion problem in terms of reflecting and absorbing boundary conditions.
The semigroups generated by the operators defined using these boundary conditions
are positivity preserving contractions on $l^1(\wt )$.

Despite the connections with probability, this work will not 
explicitly use probabilistic techniques or interpretations. 
We simply mention the classic work \cite{Feller57},and the 
recent works \cite{Fukushima,Liggett,Woess} as pointers to 
the enormous literature related to analysis of infinite state Markov chains. 

\section{Weakly connected graphs}

\subsection{Topology}

$\graph $ will denote a simple graph with a countable vertex set $\vertexset$
and a countable edge set $\edgeset$.  Each vertex 
will have at least one and at most finitely many incident edges.  
Vertices of degree $1$ are {\it boundary vertices}; the rest are {\it interior vertices}.
$\graph $ is assumed to have {\it edge weights} ({\it resistances}).  
That is, there is a function $\rest :\edgeset \to (0,\infty )$, denoted by
$\rest (u,v)$ when $[u,v] \in \edgeset $. 
General references on graphs are \cite{Chung,Diestel}.
Edge weights, considered the length of the edges, are commonly identified with 
electrical network resistance \cite{Doyle} or \cite{Lyons}, and then {\it edge conductance} 
is the reciprocal $\cond (u,v) = 1/\rest (u,v) $ if $\rest(u,v) > 0$, and $0$ otherwise.  

A {\it finite path} (sometimes called a walk) in $\graph $ connecting vertices $u$ and $v$ 
is a finite sequence of vertices
$u=v_0, v_1,\dots ,v_K = v$ such that $[v_k,v_{k+1}] \in \edgeset $ for $k = 0,\dots ,K-1$. 
$\graph $ is connected there is a finite path from $u$ to $v$ for all $u,v \in \vertexset$.  
Define a metric on (the vertices of) $\graph $ by   
\begin{equation} \label{metricdef}
d(u,v) = \inf_{\gamma } \sum_k \rest (v_{k+1},v_k), 
\end{equation}
the infimum taken over all finite paths  $\gamma $ joining $u$ and $v$.
If there is no finite path from $u$ to $v$ then $d(u,v) = \infty $.
$\graphbar $, with the extended metric $d$, will denote the metric space completion 
\cite[p. 147]{Royden} of $\graph $.  

Extending the combinatorial notion of path, a {\it path} in $\graphbar $
will be a sequence $\{ v_k \}$ with $v_k \in \vertexset $, $[v_k,v_{k+1}] \in \edgeset$,
where the index set may be finite (finite path), the positive integers 
(a ray), or the integers (a double ray).  The role of continuous paths in $\graphbar $  
is played by paths going from $u \in \graphbar $ to $v \in \graphbar $, 
which in the double ray case requires
$\lim _{k \to -\infty} d(v_k,u) = 0$ and $\lim _{k \to \infty} d(v_k, v) = 0$. 
The ray case is similar.
A path for which all vertices are distinct is a {\it simple path}.     
If $\graph $ is connected then there is a path joining any pair of points $u,v \in \graphbar $.
Modifying ideas from \cite{Carlson08}, say that $\graphbar $ is {\it weakly connected} 
if for every pair of distinct points $u,v \in \graphbar$ 
there is a finite set $W$ of edges in $\graph $ such that every path from
$u$ to $v$ contains an edge from $W$.

One may extend $\graph $ to a metric graph $\graph _m$ by identifying the combinatorial edge
$[u,v]$ with an interval of length $\rest(u,v)$.  With the usual metric on $\graph _m$,
its vertex set will be isomorphic to $\graph $.  By this device some of   
the results of \cite{Carlson08}, which should be consulted for more details, 
carry over to the present context.  
The next result is a simple example.

\begin{prop}
If $\tree$ is a tree then $\treebar $ is weakly connected.
\end{prop}

The volume of a graph is defined as the sum of its edge lengths,
\[vol(\graph) = \sum_{[v_1,v_2] \in \edgeset} \rest(v_1,v_2).\]
Finite volume graphs also have weakly connected completions \cite{Carlson08}. 

\begin{prop}
If $vol (\graph ) < \infty $ then $\graphbar $ is weakly connected.
\end{prop}

\begin{proof}
The main case considers distinct points $x$ and $y$ in $\graphbar \setminus \graph$.
Remove a finite set of edges from $\edgeset $ so that the remaining edgeset $\edgeset _1$
satisfies 
\[ \sum_{[v_1,v_2] \in \edgeset _1} \rest(v_1,v_2) < \frac{d(x,y)}{2}.\] 
Proceeding with a proof by contradiction,
suppose $(\dots ,v_{-1},v_0,v_1,\dots )$ is a path 
from $x$ to $y$ using only edges in $\edgeset _1$.
Then for $n$ sufficiently large $d(v_{-n},v_n) > \frac{d(x,y)}{2}$,
but there is a simple path from $v_{-n}$ to $v_n$ using only edges from $\edgeset _1$,
so $d(v_{-n},v_n) < \frac{d(x,y)}{2}$.
\end{proof}

The next result gives conditions allowing edges to be added
to a weakly connected graph without disturbing that property.

\begin{thm} \label{addedge}
Suppose the graph $\graph _0$, with vertex set $\vertexset$, has 
a weakly connected completion $\graphbar _0$.
Using the same vertex set $\vertexset$, enlarge $\graph _0$ to a graph $\graph _1$ 
by adding a sequence $\edgeset _1$ of edges $e_n$ whose lengths $\rest _n$ satisfy 
$\lim_{n \to \infty} \rest _n = 0$.  
Assume there is a positive constant $C$ such that 
\[d_{\graph _1}(u,v) \le d_{\graph _0}(u,v) \le C d_{\graph _1}(u,v), \quad u,v \in \vertexset.\]
Then $\graphbar _1$ is weakly connected.
\end{thm}

\begin{proof}

First note that $\graphbar _1 \setminus \graph _1 = \graphbar _0 \setminus \graph _0$,
since the set of Cauchy sequences of vertices has not changed.
It will suffice to consider distinct points $u$ and $v$ in $\graphbar _1 \setminus \graph _1$ 
which are joined by a path in $\graphbar _1$.  
Let $W_0$ be a finite set of edges in $\graph _0$ such that every path in $\graphbar _0$ from
$u$ to $v$ contains an edge $[v_1,v_2]$ from $W_0$.  

Pick $\epsilon > 0$ such that $\epsilon < \rest (v_1,v_2)$ for all edges $[v_1,v_2] \in W_0$.
Find $N$ so that the lengths $R_n$ of edges $e_n \in \edgeset _1 $ 
satisfy $\rest _n < \epsilon /C$ for $n > N$.
Let $W_1$ be the set of edges $W_0 \cup \{ e_1,\dots ,e_N \}$ in $\graph _1$. 
Suppose there is a path $\gamma $ in $\graphbar _1$ 
joining $u$ to $v$, but not containing any edge from $W_1$.  

Let $\Gamma $ be the set of edges $e_j \in \edgeset _1$ which are in $\gamma $.
$\Gamma $ is not empty since $\gamma $ contains at least one edge not in $\graph _0$.  
Since $\gamma $ contains no edge from $W_1$, the edges $e_j \in \Gamma $ have 
length $\rest _j < \epsilon /C$.  If an edge $e_j = [v_j,v_{j+1}] \in \Gamma $ 
has $d_{\graph _0}(v_j,v_{j+1}) < \epsilon $, 
then $e_j$ can be replaced by a finite path $\gamma _j$ in $\graph _0$
with length at most $CR_j$, and containing no edge from $W_0$.
Thus there is at least one $e_j = [v_j,v_{j+1}] \in \Gamma $ 
with $d_{\graph _0}(v_j,v_{j+1}) \ge \epsilon $.
But the inequalities $d_{\graph _1} (v_j,v_{j+1}) < \epsilon /C$ while 
$d_{\graph _0}(v_j,v_{j+1}) \ge \epsilon $ contradict the hypotheses, so no such path
from $u$ to $v$ exists, and $\graphbar _1$ is weakly connected.

\end{proof}

The conclusion of \thmref{addedge} may be false if the edge lengths $R_n$ 
have a positive lower bound.  Start with $\graph _0$ which is 
be weakly connected and connected.  Take distinct points 
$x,y \in \graphbar _0 \setminus \graph _0$.  Suppose $\gamma $ is a path from
$x$ to $y$.  Take sequences of vertices $x_n,y_n$ from $\gamma $  
with $ x_n \to x$ and $y_n \to y$ such that $x_n \not= y_n$, 
and $[x_n,y_n]$ is not an edge in $\graph _0$, 
Form $\graph _1$ by adding edges $[x_n,y_n]$ to $\graph _0$, with
$\rest (x_n,y_n) = d_{\graph _0}(x_n,y_n)$.
Since $d_{\graph _1}(u,v) = d_{\graph _0}(u,v)$ for all vertices $u,v$,
the vertex sets for $\graph _0$ and $\graph _1$ are isometric.  
The edge lengths $\rest (x_n,y_n)$ have a positive lower bound, and
$\graphbar _1$ is not weakly connected,    

\begin{thm} \label{disconnect}
Assume $\graphbar $ is weakly connected. 
If $U$ and $V$ are disjoint compact subsets of $\graphbar $,
then there is a finite set $W$ of edges in $\graph $ such that
every path from $U$ to $V$ contains an edge in $W$.  
\end{thm}
 
\begin{proof}
Since $\graphbar $ is weakly connected, if $u \in U$ and $v \in V$ there is a 
finite set $W(u,v)$ of edges in $\graph $ such that
every path from $u$ to $v$ contains an edge in $W$.  
Take $\epsilon > 0$ such that $\epsilon < \rest (v_1,v_2)$ for all edges $[v_1,v_2] \in W$.
If $z_1 \in B_{\epsilon }(u)$, the open $\epsilon $ ball centered at $u$,  
and $z_2 \in B_{\epsilon }(v)$, then
every path from $z_1$ to $z_2$ contains an edge in $W(u,v)$.  
The collection $\{ B_{\epsilon}(u) \times B_{\epsilon}(v), u \in U, v \in V \}$ 
is an open cover of the compact set $U\times V$, so there is a finite subcover
$B_{\epsilon _n}(u_n) \times B_{\epsilon _n}(v_n)$ for $n = 1,\dots ,N$.   
Take $W = \cup_n W(u_n,v_n)$.
\end{proof}

Suppose $W$ is a nonempty finite set of edges in $\edgeset $.
For $x \in \graphbar $, let $U_W(x)$ be the set of points $y \in \graphbar $
which can be connected to $x$ by a path containing no edge of $W$. 

\begin{lem} \label{cloplem}
For all $x \in \graphbar $ the set $U_W(x)$ is both open and closed in $\graphbar $.
\end{lem}

\begin{proof}
Take $\epsilon > 0$ such that $\epsilon < \rest (v_1,v_2)$ for all edges $[v_1,v_2] \in W$.
Suppose $y$ and $z$ are vertices with $d(z,y) < \epsilon $, so there is a path
of length smaller than $\epsilon $ from $y$ to $z$.  If there is a path from 
$y$ to $x$ containing no edge from $W$, then 
by concatenating these paths there is a path from $z$ to $x$ containing no edge from $W$.
This shows that $U_W(x)$ and the complement of $U_W(x)$ are both open.
\end{proof}

\begin{thm} \label{totally} 
A weakly connected $\graphbar $ is totally disconnected.
\end{thm}

\begin{proof}
Suppose $v_1$ and $v_2$ are distinct points in $\graphbar$, with 
$W$ being a finite set of edges in $\graph $ such that every path from
$v_1$ to $v_2$ contains an edge from $W$.    
The set $U_W(v_1)$ is both open and closed.  
Since $U_W(v_1)$ and $U_W(v_2)$ are disjoint, 
$U_W(v_2) \subset U_W^c(v_1)$, the complement of $U_W(v_1)$ in $\graphbar $. 
Thus $v_1$ and $v_2$ lie in different connected components. 
\end{proof}

If $\graphbar $ is totally disconnected and compact, it
has a rich collection of clopen sets, that is sets which are both open and closed.
In fact \cite{Carlson08} or \cite[p. 97]{Hocking} 
for any $x \in \graphbar $ and any $\epsilon > 0$
there is a clopen set $U$ such that $x \in U \subset B_{\epsilon }(x)$.  In particular
any compact subset of $\graphbar $ can then be approximated by a clopen set.

Changing the metric on $\graph $ may change the completion $\graphbar $ and
the functions on $\graph $ that extend continuously to graphbar.  
Given $\graph $ and two weight functions $R_0$ and $R_1$, let 
$\graphbar _0$ and $\graphbar _1$ denote the completions of $\graph$ with respect to the
associated metrics $d_0$ and $d_1$.
Say that the weight function $R_1(u,v)$ is smaller than the weight function $R_0(u,v)$
if $R_1(u,v) \le R_0(u,v)$ for every edge $[u,v] \in \edgeset $. 
If $R_1$ is smaller than $R_0$, the metric $d_1$ on $\graph $ extends to a 
pseudometric \cite[p. 140-141]{Royden} on $\graphbar _0$; that is, 
there may be distinct points $x,y \in \graphbar _0$ with $d_1(x,y) = 0$.
The next result establishes a stability property for weakly connected graph completions 

\begin{thm} \label{changemet}
Suppose $\graphbar _0 $ is weakly connected.
If $R_1$ is smaller than $R_0$ then the pseudometric on $\graphbar _0$
induced by $R_1$ is a metric.
If in addition $(\graphbar _0,d_0)$ is compact, then $(\graphbar _0,d_0)$ and $(\graphbar _0,d_1)$ 
are homeomorphic.
\end{thm}

\begin{proof}
The usual construction \cite[p. 147]{Royden} identifies the 
completion of a metric space with equivalence classes of Cauchy sequences,
two sequences $\{ x_n \}$ and $\{ y_n \}$ being equivalent when $d(x_n,y_n) \to 0$.
Since $\graph $ is locally finite, it is only possible to have distinct points  
$x,y \in \graphbar _0$ with $d_1(x,y) = 0$ if $x,y \in \graphbar _0 \setminus \graph $.

Suppose $x$ and $y$ are distinct points of $\graphbar _0 \setminus \graph $,
and that $W$ is a finite set of edges such that 
any path from $x$ to $y$ contains an edge from $W$. 
Let $\{ x_n \}$ and $\{ y_n \}$ be sequences in $\vertexset $ with $x_n \to x$ and $y_n \to y$.  
By \lemref{cloplem} there is an $N$ such that $n \ge N$ implies any path in $\graph $ 
from $x_n$ to $y_n$ must contain an edge from $W$.  
Let $0 < \epsilon < R_1(u,v)$ for all edges $[u,v] \in W$.  
Then $d_1(x_n,y_n) \ge \epsilon $ for all $n \ge N$, so
the sequences $\{ x_n \}$ and $\{ y_n \}$ are still inequivalent Cauchy sequences 
with respect to $d_1$, showing that $d_1$ is a metric on $\graphbar _0$.

Let $\iota $ denote the identity map from $(\graphbar _0,d_0)$ to $(\graphbar _0,d_1)$,
which is distance reducing, so continuous.  If $(\graphbar _0,d_0)$ is compact,
then so is $(\graphbar _0,d_1)$.  Following \cite[p.123]{Folland},
suppose $K \subset (\graphbar _0,d_0)$ is closed.
Then $(\iota ^{-1})^{-1}(K) = \iota (K)$ is compact, so closed, and $\iota ^{-1}$ is continuous.     
\end{proof}

\subsection{Function theory}

There is an algebra of functions with pointwise addition and multiplication
well matched to the weakly connected graph completions.
Define the 'eventually flat' functions $\alg $ to be
the real algebra of functions $\phi :\vertexset \to \real $
such that the set of edges $[u,v]$ in $\edgeset $ with 
$\phi (u) \not= \phi (v)$ is finite.   
  
\begin{lem}
All functions $\phi \in \alg $ extend continuously to $\graphbar $.
\end{lem}

\begin{proof}
Let $W$ be the finite set of edges $[u,v]$ with $\phi (u) \not= \phi (v)$.
Suppose $x \in \graphbar $, $x_n \in \graph $, and $x_n \to x$.
If $x \in \graph $ then $x_n = x$ for $n$ sufficiently large.
Suppose $x \in \graphbar \setminus \graph $.
By \lemref{cloplem} the set $U(W)(x)$ is open, so there is an $N$ such that 
$n \ge N$ implies $x_n \in U_W(x)$ and $\phi (x_n) = \phi (x_{n+1})$.
Take $\phi (x) = \lim_{n \to \infty} \phi (x_n)$.
\end{proof}

Making use of this lemma, functions $\phi \in \alg $ are extended continuously
to functions on $\graphbar $.

\begin{lem}
If $\graph $ is connected, then any $\phi \in \alg $ has finite range.
For $c \in \real $, $\phi ^{-1}(c)$ is a clopen set in $\graphbar $.
\end{lem}

\begin{proof}
Suppose $\phi \in \alg $ is not constant, and suppose $u \in \vertexset $. 
Find a path $(u = v_0,v_1,\dots ,v_N)$ such that
$\phi (v_n) = \phi (u)$ for $n \le N$ and $\phi (v_N) \not= \phi (w)$ 
for some $w$ adjacent to $v_N$.   
Since the set of such vertices $v_N$ is finite, 
$\phi (u)$ has one of a finite set of values.

For $c \in \real $, $\phi ^{-1}(c)$ is a closed set, and its complement in $\graphbar $
is the union of a finite collection of closed sets.

\end{proof}

The next result shows that $\alg $ separates points of $\graphbar $ if and only
if $\graphbar $ is weakly connected.

\begin{thm} \label{separate}
Suppose $\graphbar $ is weakly connected.  If $x$ and $y$ are distinct points of
$\graphbar $, there is a function $\phi \in \alg $ whose range is $\{ 0 ,1 \}$
such that $\phi (z) = 0$ for $z$ in an open neighborhood $U$ of $x$, and  
$\phi (z) = 1$ for $z$ in an open neighborhood $V$ of $y$. 

Conversely, if $\alg $ separates points of $\graphbar $, then 
$\graphbar $ is weakly connected. 
\end{thm}

\begin{proof}
Let $W$ be a finite set of edges in $\graph $ such that every path from $x$ to $y$ 
contains an edge from $W$.  By \lemref{cloplem} the set $U_W(x)$ is both open and
closed in $\graphbar $, as is $U_W^c(x)$
Define $\phi (z) = 0$ for $z \in U_W(x)$ and $\phi (z) = 1$ for $z \in U_W^c(x)$.
For every vertex $v$ and adjacent vertex $w$ we have $\phi (v) = \phi (w)$ 
unless $[v,w] \in W$.  Since $W$ is finite, $\phi \in \alg $.

In the other direction, suppose $\alg $ separates points of $\graphbar $.
Let $x$ and $y$ be distinct points in $\graphbar $, and suppose $\phi \in \alg $
with $\phi (x) < \phi (y)$.  Let $W$ be the finite set of edges $[u,v]$
such that $\phi (u) \not= \phi (v)$.  If $\gamma $ is any path starting at $x$ which 
contains no edge from $W$, then $\phi $ must be constant along $\gamma $.
That is, every path from $x$ to $y$ must contain an edge from $W$, so
$\graphbar $ is weakly connected. 

\end{proof}

\begin{cor} \label{sepsets}
Suppose $\graphbar$ is weakly connected.  If $\Omega $ and $\Omega _1$
are nonempty disjoint compact subsets of $\graphbar $, then there is a function
$f \in \alg $ such that $0 \le f \le 1$,
\[f(x) = 1, \quad x \in \Omega , \quad f(y) = 0, \quad y \in \Omega _1.\]
\end{cor}

\begin{proof}
First fix $y \in \Omega _1$.
Using \thmref{separate}, find a finite cover $U_1,\dots ,U_N$ of $\Omega $ 
by open sets with corresponding functions $f_1,\dots ,f_N$ which satisfy
$f_n(z) = 1$ for all $z$ in some open neighborhood $V_y$ of $y$, while 
$f_n(z) = 0$ for all $z$ in $U_n$.  

Define $F_y = f_1\cdots f_N$.
Find a finite collection $F_1,\dots, F_N$ whose corresponding open sets
$V_1,\dots ,V_N$ cover $\Omega _1$.  The function 
$f = (1-F_1)\cdots (1-F_N)$ has the desired properties.
\end{proof}

The combination of \thmref{separate} and the Stone-Weierstrass Theorem yields the next result.

\begin{thm} \label{Stone}
If $\graphbar $ is weakly connected and compact, then $\alg $ is uniformly
dense in the continuous functions on $\graphbar $.
\end{thm}

\section{Quadratic forms}

\subsection{Weights and forms}

Vertex weights and the corresponding measure are now added to the edge weighted graph $\graph $.
A vertex weight function $\wt : \vertexset \to (0,\infty)$ provides the Borel measure 
$\mu (U) = \sum_{v \in U} \wt (v)$ on $\graph $, which may be extended \cite[p. 257]{Royden}
to $\graphbar $ by defining the measure of $\graphbar \setminus \graph $ to be $0$.  
The real Hilbert space $l^2(\wt )$ will consist of functions 
$f:\vertexset \to \real $ with $\sum_{v \in \vertexset} f(v)^2 \wt (v) < \infty $, 
and inner product $\langle f,g \rangle _{\wt} = \sum_{v \in \vertexset} f(v)g(v)\wt (v)$.  
The set of functions $f:\vertexset \to \real $ which are $0$ 
at all but finitely many vertices is denoted by $\domain _K$.
Also introduce $\domain _{\alg ,\wt } = \alg  \cap l^2(\wt )$.

The next proposition collects basic facts about the symmetric bilinear forms 
induced by the edge conductances $C(u,v)$.
Closely related results using the smaller domain $\domain _K$
are in \cite[p. 20]{Davies} and \cite{KL2}.

\begin{prop} \label{bform}
Suppose $C:\vertexset \times \vertexset \to [0,\infty) $ satisfies
$C(u,v) = C(v,u)$, with $C(u,v) > 0$ if and only if $[u,v] \in \edgeset$.  
Given a vertex weight $\wt $, the symmetric bilinear form
\begin{equation} \label{bformdef}
B(f,g) = \frac{1}{2}\sum_{v \in \vertexset} \sum_{u \sim v} C(u,v)(f(v) - f(u))(g(v) - g(u)),
\quad f,g \in \domain _{\alg ,\wt}, 
\end{equation}
has a nonnegative quadratic form $B(f,f)$, and satisfies
\begin{equation} \label{symmetry}
B(f,g) = \langle \Delta _{\wt} f,g \rangle _{\wt} = \langle f, \Delta _{\wt}g \rangle _{\wt},
\end{equation}
where
\begin{equation} \label{Laplace}
\Delta _{\wt}f(v) = \frac{1}{\wt (v)}\sum_{u \sim v} C(u,v) (f(v) - f(u)).
\end{equation}

\end{prop}

\begin{proof}

The nonnegativity of the quadratic form is immediate from the definition.
Note that for any $f \in \alg $ there are only finitely many vertices $v \in \vertexset $ 
for which $f(v) - f(u)$ is nonzero if $u$ is adjacent to $v$.

To identify the operator $\Delta _{\wt}$, start with
\begin{equation} \label{compform}
2B(f,g) 
 = \sum_{v \in \vertexset} g(v) \sum_{u \sim v} C(u,v) \Bigl ( f(v) - f(u) \Bigr )
\end{equation}
\[ - \sum_{v \in \vertexset} \Bigl ( \sum_{u \sim v} C(u,v) g(u)(f(v) - f(u)) \Bigr )\]
Suppose a graph edge $e$ has vertices $v_1(e)$ and $v_2(e)$.  
The second sum over $v \in \vertexset$ in \eqref{compform} can be viewed as a 
sum over edges, with each edge contributing the terms
$C(v_1,v_2)g(v_1)(f(v_2) - f(v_1))$ and $C(v_1,v_2)g(v_2)(f(v_1) - f(v_2))$.
Using this observation to change the order of summation gives
\[ \sum_{v \in \vertexset} \Bigl ( \sum_{u \sim v} C(u,v) g(u)(f(v) - f(u)) \Bigr )\]
\[ = \sum_{e \in \edgeset} C(v_1(e),v_2(e))
\Bigl (g(v_1)(f(v_2) - f(v_1)) + g(v_2)(f(v_1) - f(v_2)) \Bigr )\]
\[ = \sum_{u \in \vertexset} g(u) \sum_{v \sim u} C(u,v) (f(v) - f(u))  \]
Employing this identity in \eqref{compform} gives
\begin{equation} \label{formform}
2B(f,g) 
= 2 \sum_v \wt (v) g(v) [\frac{1}{\wt (v)}\sum_{u \sim v} C(u,v) (f(v) - f(u))].
\end{equation}

\end{proof}

With respect to the standard basis consisting of functions $\delta _w: \vertexset \to \real $
with $\delta _w(w) = 1$ and $\delta _w(v) = 0$ for $v \not= w$, 
the operators $\Delta _{\wt}$ have the matrix representation 
\[Q(v,w) = \Bigl \{ \begin{matrix} \wt ^{-1}(w) \sum_{u \sim w} C(u,w), & v = w \cr 
-\wt ^{-1}(v) C(v,w), & v \sim w \cr 
0, & {\rm otherwise } \end{matrix} \Bigr \},
\quad v,w \in \vertexset . \]
If $v$ is fixed, then summing on $w$ gives
\[\sum_w Q(v,w) = \frac{1}{\wt (v)} \sum_{u \sim v} C(u,v) 
- \frac{1}{\wt (v)} \sum_{u \sim v} C(u,v) = 0,\]
so $-Q(v,w)$ is a $Q$-matrix in the sense of Markov chains \cite[p. 58]{Liggett}.

In the $Q$ - matrix formulation the matrix entries represent transition rates,
so decreasing the vertex measure $\wt $ increases the rates. 
Consistent with the boundary value themes of this work an interesting choice is 
to take the vertex weight $\wt _0(v)$ to be half the sum 
of the lengths of the incident edges,
\[\wt _0(v) = \frac{1}{2}\sum_{u \sim v} R(u,v) .\]
This choice makes the vertex measure consistent with the previously defined
graph volume,
\[\wt _0(\graph ) = \sum_{e \in \edgeset} R(e) = {\rm vol (\graph )} .\]
If $\rm vol (\graph ) < \infty $, then $l^2(\wt _0)$ will include all functions in $\alg$. 

With respect to this measure 
\[\Delta _{\wt _0 }f(v) = \wt _0^{-1}(v)\sum_{u \sim v} C(u,v) (f(v) - f(u))
= \frac{2}{\sum_{u \sim v} R(u,v)} \sum_{u \sim v} \frac{f(v) - f(u)}{R(u,v)}.\]
This operator resembles the symmetric second difference operator from 
numerical analysis.  Like the classical Laplace operator in Euclidean space, 
$\Delta _{\wt _0}$ exhibits quadratic scaling behavior.  
Suppose $v_0,v_1 \in \vertexset $ have equal numbers
of incident edges, and there is a bijection of vertex neighborhoods   
with $R(v_1,w_1) = \rho R(v_0,w_0)$ for $\rho > 0$, and  $w_j \sim v_j$.  
If $f:\vertexset \to \real $ satisfies $f(v_1) = f(v_0)$ and 
$f(w_1) = f(w_0)$, then
\[\Delta _{\wt _0}f(v_1) = \rho ^{-2}\Delta _{\wt _0}f(v_0) .\]

The vertex weight $\wt _0$ is typically distinct from 
$\wt (v) = \sum_{u \sim v} C(u,v)$, a choice 
which appears in the study of discrete time Markov chains
\cite[p. 40]{Doyle}, \cite[p. 73]{Liggett}, \cite[p. 18]{Lyons} 
with transition probabilities $p(u,v) = \wt ^{-1}(v) C(u,v)$ for $ u \not= v$.

\subsection{Continuity}

The next result considers continuous extension of functions to $\graphbar $ 
when the quadratic form is finite. 
 
\begin{thm} \label{formcont}
Suppose $\graph $ is connected. 
Using the metric of \eqref{metricdef}, 
functions $f:\vertexset \to \real$ with $B(f,f) < \infty $
are uniformly continuous on $\graph $, and 
so $f$ extends uniquely to a continuous function on $\graphbar $.   
\end{thm}

\begin{proof}
If $v,w \in \vertexset  $ and $\gamma = (v=v_0,v_1,\dots ,v_K = w) $ 
is any finite simple path from $v$ to $w$, then the Cauchy-Schwarz inequality gives
\[|f(w) - f(v)|^2 
 = \Big | \sum_{k} [f(v_{k+1}) - f(v_{k})] 
\frac{C^{1/2}(v_{k+1},v_k)}{C^{1/2}(v_{k+1},v_k)} \Big |^2 \]
\[ \le \sum_{k} [C(v_{k+1},v_k)(f(v_{k+1}) - f(v_{k}))^2] 
\sum_k R(v_{k+1},v_k) \]
\[ \le 2B(f,f) \sum_k R(v_{k+1},v_k). \]
There is a simple path with $\sum_k R(v_{k+1},v_k) \le 2d(v,w)$, so
\begin{equation} \label{formest}
|f(w) - f(v)|^2 \le 4B(f,f)d(v,w), 
\end{equation}
which shows $f$ is uniformly continuous on $\graph $.
By \cite[p. 149]{Royden} $f$ extends continuously to $\graphbar $.   

\end{proof}

The bilinear form may be used to define a 'Sobolev style'
Hilbert space $H^1(\wt )$ with inner product 
\[\langle f,g \rangle _{\wt ,1} = \sum_v f(v)g(v) \wt (v) + B(f,g).\]
Let $H^1_0(\wt)$ be the closure of $\domain _K$ in $H^1(\wt)$.

\begin{lem} \label{h1lem}
If $\graph $ is connected with finite diameter, then there is a constant $C$
such that 
\begin{equation} \label{conteval}
\sup _{v \in \vertexset} |f(v)| \le C\| f \| _{\wt ,1},
\end{equation}
so a Cauchy sequence in $H^1(\wt )$ is a uniform Cauchy sequence.   The functions
$f$ in the unit ball of $H^1$ are uniformly equicontinuous \cite[p. 29]{Royden}.
\end{lem}

\begin{proof}
Fixing a vertex $v_0$, \eqref{formest} gives
\[|f(v)| \le |f(v_0)| + |f(v) - f(v_0)|  \le
\| f \| _{\wt ,1} /\sqrt{\wt (v_0)} + 2\| f \| _{\wt ,1} {\rm diam}(\graph )^{1/2},\]
which is \eqref{conteval}.  The uniform equicontinuity follows  
from \eqref{formest}.
\end{proof}

\begin{thm} \label{zerobnd}
Suppose $\graph $ is connected and has finite diameter.
If $f \in H^1_0(\wt )$, then $f$ has a unique continuous extension to $\graphbar$
which is zero at all points $x \in \graphbar \setminus \graph $. 
\end{thm}

\begin{proof}
Any function $f \in H^1_0(\wt )$ is the limit in $H^1(\wt )$ of a 
sequence $f_n$ from $\domain _K$.  The functions $f$ and $f_n$ have unique continuous 
extensions to $\graphbar$ by \thmref{formcont}.  The extended functions $f_n$ 
satisfying $f_n(x) = 0$ for all $x \in \graphbar \setminus \graph $. 
By \lemref{h1lem} the sequence $f_n$ converges to $f$ uniformly on $\graph $, 
so the extensions $f_n$ converge uniformly to the extension $f$ on $\graphbar$.  
Thus $f(x) = 0$ for all $x \in \graphbar \setminus \graph $. 
\end{proof}

\section{Boundary value problems}

\subsection{The basic Laplacian}

Let $S_{K,\wt }$ denote the operator 
$\Delta _{\wt}$ on $l^2(\wt )$ with the domain $\domain _K$.  

\begin{prop} \label{Lapdef}
The operator $S_{K,\wt}$ is symmetric and nonnegative on $l^2(\wt)$.
The adjoint operator $S_{K,\wt}^*$ on $l^2(\wt )$ acts by
\[(S_{K,\wt}^*h)(v) = \Delta _{\wt }h(v) = \frac{1}{\wt (v)} \sum_{u \sim v} C(u,v) (h(v) - h(u))\]
on the domain consisting of all $h \in l^2(\wt )$ for which 
$\Delta _{\wt }h \in l^2(\wt )$.  
\end{prop}

\begin{proof}
The symmetry and nonnegativity of $S_{K,\wt}$ are given by \eqref{symmetry}.  
Since $S_{K,\wt}$ is densely defined, $S_{K,\wt}^*$ is the operator
whose graph is the set of pairs 
$(h,k) \in l^2(\wt ) \oplus l^2(\wt )$ 
such that
\[\langle S_{K,\wt}f,h \rangle _{\wt } =  \langle f,k \rangle _{\wt } \]
for all $f \in \domain _K$.  Suppose $f_v = \frac{1}{\wt (v)}\delta _v$
Then for any $h$ in the domain of $S_{K,\wt}^*$,
\[k(v) = (S_{K,\wt}^*h)(v) = \langle S_{K,\wt}f_v,h \rangle _{\wt } 
 = \sum _w [\sum_{u \sim w} C(u,w) (f_v(w) - f_v(u))]h(w)\] 
\[ = \frac{1}{\wt (v)} \sum_{u \sim v} C(u,v) (h(v) - h(u)).\]

\end{proof}

\propref{Lapdef} provides a basic Laplace operator,
the Friedrich's extension \cite[pp. 322-326]{Kato} of $S_{K,\wt}$,
whose domain is a subset of $H^1_0(\wt)$ the closure of 
$\domain _K$ in $H^1(\wt)$.
Let $\dop _{K,\wt}$ denote the Friedrich's extension of $S_{K,\wt }$.
Several features of $\dop _{K,\wt}$ are implied by the condition
$\wt (\graph ) < \infty $.  

\begin{prop} \label{lbnd}
Suppose $\graph $ is connected, with finite diameter and infinitely many vertices. 
If $\wt (\graph ) < \infty $, $f \in {\rm domain}(\dop _K)$, and
$\| f \| _{\wt } = 1$,  
then $\dop _{K,\wt}$ has the strictly positive lower bound
\begin{equation}
\langle \dop _{K,\wt} f,f \rangle _{\wt } = B(f,f) \ge \frac{1}{4\wt (\graph ) {\rm diam} (\graph)}, 
\end{equation}
\end{prop}

\begin{proof}
The Friedrich's extension $\dop _{K,\wt}$ of the nonnegative symmetric operator $S_{K,\wt}$
has the same lower bound, so it suffices to consider functions $f \in \domain _K$.
Since $\| f \| _{\wt} = 1$ there must be some vertex $v$ where $f^2(v) \ge \wt ^{-1}(\graph )$.
Since $f$ has finite support, there is another vertex $u$ with $f(u) = 0$.
An application of \eqref{formest} gives
\[\wt ^{-1}(\graph ) \le f^2(v) = [f(v) - f(u)]^2 \le 4B(f,f)d(u,v).\]
\end{proof}

\begin{prop}
Suppose $\graph $ is connected, $\graphbar $ is compact, and $\wt (\graph )$ is finite.
Let $S_{1,\wt}$ be a symmetric extension of $S_{K,\wt}$ in $l^2(\wt )$ 
whose associated quadratic form is 
\[ \langle S_{1,\wt}f,f \rangle _{\wt } = B(f,f).\]
Then the Friedrich's extension $\dop _{1,\wt}$ of $S_{1,\wt}$ has compact resolvent.
\end{prop} 

\begin{proof}
The resolvent of $\dop _{1,\wt}$ maps a bounded set in $l^2(\wt )$ into a bounded set
in $H^1(\wt )$. Suppose $f_n$ is a bounded sequence in $l^2(\wt )$,
with $g_n = (\dop _1 - \lambda I)^{-1}f_n$.  
By \lemref{h1lem} and the Arzela-Ascoli Theorem \cite[p. 169]{Royden} the sequence $g_n$
has a uniformly convergent subsequence, which converges in $l^2(\wt )$.
\end{proof}

\subsection{Harmonic functions}

When considering harmonic functions, it will be convenient to treat vertices of 
degree one (boundary vertices) as part of the boundary of $\graphbar $.
For instance, this convention allows finite graphs with degree one vertices
to have nonconstant harmonic functions.  
Let $\vertexset _{int}$ denote the set of interior vertices, that is vertices
with degree at least $2$.
Say that a function $f:\vertexset \to \real$ is {\it harmonic on $\graph $} if 
\begin{equation} \label{harmdef}
f(v) = \frac{1}{\sum_{u \sim v}C(u,v) }\sum_{u \sim v} C(u,v) f(u)
\end{equation}
for all vertices $v \in \vertexset _{int}$.
Since the value $f(v)$ of a harmonic function is the weighted
average of the values at the adjacent vertices, $f$ has a minimum or maximum
at some $v \in \vertexset _{int}$ of a connected graph if and only if $f$ is constant.   
The set of harmonic functions is independent of the vertex weight $\wt$, 
but harmonic functions may be described as those for which
\[\Delta _{\wt }f(v) = \frac{1}{\wt (v)} \sum_{u \sim v} C(u,v) (f(v) - f(u)) = 0, \quad 
v \in \vertexset _{int} .\]

\begin{prop}
Suppose $\graph $ is connected and $\graphbar$ is compact. 
If $\graphbar \setminus \graph \not= \emptyset $, 
then $0$ is not an eigenvalue of an operator $\dop _{K,\wt}$.
\end{prop}

\begin{proof}
If $0$ were an eigenvalue of $\dop _{K,\wt}$, then there would be a corresponding
eigenfunction $f$ which is positive somewhere.
By \thmref{zerobnd}, $f$ would extend continuously to $\graphbar$, with
$f(x) = 0$ for all $x \in \graphbar \setminus \graph $. 
Since $\graphbar $ is compact, $f$ has a positive maximum at some $x \in \graph $.
Since \eqref{harmdef} holds at all vertices, $f$ is constant, but then
$\graphbar \setminus \graph \not= \emptyset $ implies $f = 0$. 
\end{proof}

Define $\partial \graphbar $ to be the complement of $\vertexset _{int}$ in $\graphbar $.  
The Dirichlet problem asks whether every continuous function 
$F: \partial \graphbar \to \real $ has a continuous extension $f:\graphbar \to \real$ 
which is harmonic on $\graph $.  The following example shows that 
the Dirichlet problem is not always solvable.

As shown in Figure A, construct a graph $\graph $ beginning with
a sequence of vertices $v_n$ and with edges $(v_n,v_{n+1})$, n = 0,1,2,\dots .
Take $R(v_n,v_{n+1}) = 2^{-n-1}$.
For each of the vertices $v_n$ introduce $2^n$ additional vertices $u_{n,k}$,
and edges $(v_n,u_{n,k})$ of length $1$. 
In this example, $\partial \graphbar$ consists of the boundary vertices $u_{n,k}$ 
together with the boundary point at $v = \lim_n v_n$.  
All of these points are isolated in $\partial \graphbar$.  

Suppose the function $F$ satisfies
\[ F(u_{n,k}) = 0, \quad F(v) = 1. \]
A computation with \eqref{harmdef} shows that any harmonic extension $f$ which is continuous
on $\graphbar$ must satisfy the contradictory requirements 
\[\lim_{n \to \infty} f(v_n) = 1, \quad
\lim_{n \to \infty} f(v_n) = 3/4.\]
 
\centerline{\includegraphics[height=2in,width=5in]{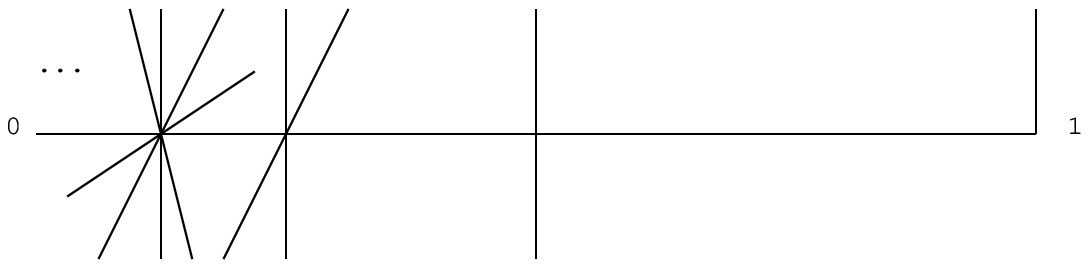}}
\centerline{Figure A} 

The graph $\graph $ in this example is connected with finite diameter, 
$\graphbar $ is weakly connected, but $\graphbar$ is not compact.
If $\graphbar$ is required to be compact, there is a positive result.

\begin{thm} \label{harm1}
Suppose $\graph $ is connected, while
$\graphbar$ is weakly connected and compact.  
Every continuous function $F:\partial \graphbar \to \real $ 
has a unique extension to a continuous function
$f: \graphbar \to \real $ that is harmonic on $\graph $. 
\end{thm}

\begin{proof}
The proof breaks into two main parts.  For the first part, assume that 
$\graph $ has no boundary vertices, so $F$ is defined on $\graphbar \setminus \graph $.
The second part of the proof will extend this partial result to the full theorem.

The uniqueness is a standard consequence of the maximum principle since the difference
of two solutions would be $0$ on $\graphbar \setminus \graph $.  
By the Tietze extension theorem \cite[p. 179]{Royden}, 
$F$ may be extended to a continuous function on $\graphbar $.
By \thmref{Stone}, for any $n > 0$ there is a $g_n \in \alg $ with
\[\max_{x \in \graphbar} |F(x) - g_n(x)| \le 1/n.\]

Pick vertex weights $\wt $ with $\wt(\graph ) < \infty $.  
Since $g_n \in \alg $, the function $\Delta _{\wt} g_n$
has the value $0$ except at a finite set of vertices.  
The compactness of $\graphbar $ means that $\graph $ has finite diameter.
By \propref{lbnd} the operator $\dop _{K,\wt }$ has a strictly positive lower bound, 
so the equation 
\[-\dop _{\wt} h_n = \Delta _{\wt} g_n, \quad h_n \big | _{\graphbar \setminus \graph } = 0, \]
has a solution.  The function
$f_n = g_n + h_n$ is harmonic on $\vertexset $, continuous on $\graphbar $,
and satisfies   
\[\max_{x \in \graphbar \setminus \graph } |F(x) - f_n(x)| \le 1/n.\]

Since $\{ f_n \}$ converges uniformly on $\graphbar \setminus \graph $, the maximum
principle implies it is a uniformly Cauchy sequence on $\graphbar $, so
there is a continuous limit $f$ on $\graphbar $.  At each $v \in \vertexset $
\[f(v) = \lim_{n \to \infty} f_n(v) 
= \lim_{n \to \infty} \frac{1}{\sum_{u \sim v}C(u,v) }\sum_{u \sim v} C(u,v) f_n(u)\]
\[ = \frac{1}{\sum_{u \sim v}C(u,v) }\sum_{u \sim v} C(u,v) f(u),\]
so $f$ is harmonic on $\graph $.  This completes the first main part of the proof.    

The second main step involves extending the theorem to include boundary vertices.
For each edge $e = [v_0,v_b]$ with boundary vertex $v_b$ and edge length $r_e = C(v_0,v_b)^{-1}$,
delete the vertex $v_b$ and edge $e$ from $\graph $.  
For $n = 1,2,3,\dots $, add a sequence of new vertices $v_n$ 
and edges $e_n = [v_n,v_{n+1}]$.  Assume the edges $e_n$ have length $r_n$ with $\sum r_n = r_e$.
The resulting graph $\graph _1$ will have a completion $\graphbar _1$
containing a point $w = \lim_{n \to \infty} v_n$.

Suppose a continuous function $G: \graphbar _1 \setminus \graph _1 \to \real $
is given, and is extended to the harmonic function $g$ on $\graph _1$.  
With $C(v_n,v_{n+1}) = r_n^{-1}$ we have
\[ g(v_n) = \frac{1}{1/r_{n-1} + 1/r_n}[
\frac{1}{r_{n-1}}g(v_{n-1}) + \frac{1}{r_n}g(v_{n+1})]\]
\[ = \frac{r_n}{r_{n-1} + r_n}g(v_{n-1}) + \frac{r_{n-1}}{r_{n-1} + r_n}g(v_{n+1}),\]
or
\[(g(v_n) - g(v_{n-1}))/r_{n-1} =  (g(v_{n+1}) - g(v_{n}))/r_{n},\]
and so $g$ is a linear function of the distance along the path from $v_0$ to $w$. 

At $v_0$ the fact that $g$ is harmonic means
\[\sum_{u \sim v_0} C(u,v_0) (g(v_0) - g(u)) = 0.\]
The linearity of $g$ along the path means that  
\[C(v_1,v_0) (g(v_0) - g(v_1)) = r_0^{-1}(g(v_0) - g(v_1)) 
= r_e^{-1}(g(v_0) - g(w)).\] 
Now define the function $f:\graph \to \real $ by
$f(v) = g(v)$ for $v \in \graph _1 \cap \graph $,
and $f(v_b) = g(w)$.  The function $f$ is harmonic on $\graph $,
and extends to the continuous function $F:\graphbar \to \real$ where
$F(x) = G(x)$ for $x \in \graphbar \setminus \graph $, while 
$F(v_b) = G(w)$ for each boundary vertex $v_b$.

\end{proof}
 
\subsection{Boundary conditions and operators}

In this section absorbing and reflecting boundary conditions are used   
to construct distinct nonnegative self adjoint extensions of
$S_{K,\wt }$.  The constructed operators extend to semigroup generators
which are positivity preserving contractions on $l^1(\wt )$.
 
Given a closed set $\Omega \subset \{ \graphbar \setminus \graph \}$, 
let $\alg _{\Omega }$ denote the
subalgebra of $\alg $ vanishing on $\Omega $.    
Define the domain
\begin{equation} \label{domdef}
\domain _{\Omega } = \alg _{\Omega } \cap l^2(\wt). 
\end{equation}
Let $S_{\Omega ,\wt}$ denote the operator with domain $\domain _{\Omega }$
acting on $l^2(\wt)$ by $S_{\Omega }f = \Delta _{\wt }f$. 

By \propref{bform} the operator $S_{\Omega ,\wt }$ is nonnegative and symmetric, 
with quadratic form $\langle S_{\Omega ,\wt} f,f \rangle _{\wt} = B(f,f)$.   
Let $\dop _{\Omega ,\wt}$ denote the Friedrich's extension of $S_{\Omega ,\wt }$,
and note that the domain of $\dop _{\Omega ,\wt}$ is a subset of $H^1(\wt )$.
A slight modification of the proof of \thmref{zerobnd} shows that every function $f_j$
in the domain of $\dop _{\Omega , \wt}$ extends continuously to $\graphbar $
with $f_j(x) = 0$ for $x \in \Omega $.  

Given a compact set $\Omega \subset \graphbar$, let 
$N_R(\Omega ) = \{ z \in \graphbar | \ d(z,\Omega ) \le R \}$
be the set of points whose distance from $\Omega$ is at most $R$.

\begin{thm} \label{distinct}
Suppose $\graphbar $ is connected, weakly connected, and compact.
Assume that for $j = 1,2$ the sets $\Omega (j) \subset \{ \graphbar \setminus \graph \}$ are compact,
and that $\wt (N_R(\Omega (j))) < \infty $ for some $R > 0$. 
If $\Omega (1) \not= \Omega (2)$,
then $\dop _{\Omega (1),\wt }$ and $\dop _{\Omega (2),\wt }$ have distinct domains.
\end{thm}

\begin{proof}
Reversing the roles of $\Omega (1)$ and $\Omega (2)$ if necessary,
we may assume $\Omega (1) \not\subset \Omega (2)$.
Let $\Omega (2)^c$ denote the complement of $\Omega (2)$.
Find $x \in \Omega (1) \cap \Omega (2)^c$, and an $r$ with $0 < r < R$ such that
$N_r(x) \subset \Omega (2)^c$.  

The sets $\{ x \} $ and $V = \{ y \in \graphbar | d(x,y) \ge r \}$ are disjoint and compact.
\corref{sepsets} shows there is a function $f \in \alg $
with $f(x) = 1$, while $f(z) = 0$ for all $z \in V$.  
Since $\wt (N_r(x)) < \infty $, the function $f$ is in the domain of 
$\dop _{\Omega (2), \wt }$, but $f$ is not in the domain of $\dop _{\Omega (1), \wt }$. 
\end{proof}

Since the operators $\dop = \dop _{\Omega ,\wt }$ are nonnegative self adjoint on
$l^2(\wt )$, the operators of the semigroup $\exp (-t\dop _{\Omega , \wt})$ 
are $l^2(\wt)$ contractions for $t \ge 0$.
For applications to probability this is not sufficient. 
Let $Quad(\dop )$ denote the domain of $\dop ^{1/2}$.
The method of \cite[p. 20]{Davies}
will show that the quadratic forms 
\[Q(f) = \langle \dop ^{1/2}f , \dop ^{1/2}f \rangle _{\wt} ,\quad f \in Quad (\dop )\]
associated to $\dop _{\Omega ,\wt }$ are Dirichlet forms.  

There are two conditions to check.  
The first condition is that $ f \in Quad (\dop ) $ implies 
$|f| \in Quad (\dop )$ and $B(|f|,|f|) \le B(f,f)$.
Since the form is
\[B(f,f) = \frac{1}{2}\sum_{v \in \vertexset} \sum_{u \sim v} C(u,v)(f(v) - f(u))^2,\]
the first condition holds for $f \in \domain _{\Omega }$.  If $f \in Quad (\dop )$ then
there is a sequence $f_n \in \domain _{\Omega }$ with
\[ \langle f, f \rangle _{\wt} + B(f,f) = \lim_{n \to \infty} 
 \langle f_n, f_n \rangle _{\wt} + B(f_n,f_n).\]
It follows that 
\[ \langle |f|, |f| \rangle _{\wt} + B(|f|,|f|) = \lim_{n \to \infty} 
 \langle |f_n|, |f_n| \rangle _{\wt} + B(|f_n|,|f_n|),\]
and $B(|f|,|f|) \le B(f,f)$.

The second condition is that if $f \in Quad(\dop )$ and $g \in l^2(\wt)$
with $|g(v)| \le |f(v)|$ and $|g(v) - g(u)| \le |f(v) - f(u)|$ for all $u,v \in \vertexset$,
then $g \in Quad(\dop )$ and $Q(g) \le Q(f)$. This is even more transparent than the first
condition.

Again quoting \cite[p. 12-13]{Davies}, the following result is established.

\begin{thm}
For $t \ge 0$ the semigroups $\exp (-\dop _{\Omega ,\wt}t)$ on $l^2(\wt )$ 
are positivity preserving contractions on $l^p(\wt )$ for $1 \le p \le \infty $ 
\end{thm}

Arguing by analogy, one expects the operator $\dop = \dop _{\Omega ,\wt}$ to 
exhibit 'absorption' at $\Omega $ and 'reflection' at $\Omega ^c$.
This analogy may be tested by estimating the rate of decay of
a probability density function initially supported near the 'reflecting' boundary.  

Suppose $\graphbar $ is compact, connected and weakly connected.
Using \thmref{totally} and the subsequent remarks, let $V$ be a 
clopen neighborhood of $\Omega $, with $U = V^c$.  The edge boundary of $U$
will be the set $\partial _e U$ of edges $e = [u,v] \in \graph $ with 
$ u \in U$ and $v \in V$.  

The edge boundary $\partial _e U $ is a finite
set since by \thmref{disconnect} there is a finite set of edges $W$ such that any path 
from $U$ to $V$ contains an edge from $W$, and $\partial _e U \subset W$.
Let $1_U$ be the indicator function of $U$, with $1_U(x) = 1$ for $x \in U$ 
and $1_U(x) = 0$ for $x \in V$.  The function $1_U$ is in $\alg $.
If $\wt (V) < \infty $ then $1_U $ is in the domain of $\dop $.

Suppose $p_0:\vertexset \to [0,\infty )$ is a bounded probability density 
with $p_0(v) \ge 0$ and $\sum_{\graph } p_0(v)\wt (v) = 1$.
For $t \ge 0$ define $p(t,v) = \exp(-t\dop )p_0(v)$ and 
denote the density integrated over $U$ by $P_U(t) = \sum_{v\in U} p(t,v)\wt (v)$. 

Using $p_0 \in l^2(\wt )$, we find that for $t > 0$,
\[\frac{d}{dt}P_U(t) 
= \frac{d}{dt} \langle \exp (-t\dop )p_0, 1_U \rangle _{\wt}
 = - \langle \dop \exp (-t\dop )p_0, 1_U \rangle _{\wt}\]
\[ = -\langle \exp (-t\dop )p_0, \dop 1_U \rangle _{\wt}
= - \sum_{(u,v) \in \partial _e U} C(u,v) (p(t,u) - p(t,v)) .\]

Since the semigroup is positivity preserving and a contraction on $l^{\infty}(\wt)$,
\begin{equation} \label{decaybnd}
\frac{d}{dt}P_U(t) \ge  
- \| p_0 \| _{\infty } \sum_{(u,v) \in \partial _e U} C(u,v) .
\end{equation}
That is, the decay rate for $P_U(t)$ 
is controlled by what happens at $\partial _e U$, without regard to
the rest of the boundary of $U$.

\newpage

\bibliographystyle{amsalpha}

\end{document}